\documentclass{ifacconf}

\usepackage{graphicx}      % include this line if your document contains figures
\usepackage{natbib}        % required for bibliography
\usepackage{amsmath,amsfonts,amssymb}
\usepackage{url}
\usepackage{multirow}
\usepackage[utf8]{inputenc}
\usepackage{graphicx}
\usepackage{xcolor}

\usepackage[baseline]{euflag}

\def\M{\mathbf{M}}

\def\N{\mathbb{N}}
\def\R{\mathbb{R}}
\def\C{\mathbb{C}}

\def\Z{\mathbb{Z}}

\def\I{\mathbf{I}}

\DeclareMathOperator{\tr}{tr}

\def\cH{\mathcal{H}}
\def\cB{\mathcal{B}}
\def\cA{\mathcal{A}}

\usepackage{braket}

\newcommand*{\unit}[1]{\operatorname{U}_{#1}(\C)}

\DeclareMathOperator{\Cstarred}{C^\star_{red}}
\DeclareMathOperator{\Cstarfull}{C^\star_{full}}

\newtheorem{theorem}{Theorem}

\newtheorem{corollary}[theorem]{Corollary}

\if{
\usepackage[colorlinks,linkcolor = blue, urlcolor = blue, citecolor = blue]{hyperref}
}\fi

\begin{document}
\begin{frontmatter}

\title{Upper bound hierarchies for noncommutative polynomial optimization \thanksref{footnoteinfo}} 

\thanks[footnoteinfo]{This work was supported by the Slovenian Research Agency program P1-0222 and grants J1-50002, J1-2453, N1-0217,
J1-3004, the NSF grant DMS-1954709, the EPOQCS grant funded by the LabEx CIMI (ANR-11-LABX-0040), the FastQI grant funded by the Institut Quantique Occitan, the HORIZON–MSCA-2023-DN-JD of the European Commission under the Grant Agreement No 101120296 (TENORS). This work was performed within the project COMPUTE, funded within the QuantERA II Programme that has received funding from the EU’s H2020 research and innovation programme under the GA No 101017733. {\normalsize\euflag}}

\author[First]{Igor Klep} 
\author[Second]{Victor Magron}
\author[Third]{Ga\"el Mass\'e} 
\author[Fourth]{Jurij Vol{\v c}i{\v c}}

\address[First]{Faculty of Mathematics and Physics, Department of Mathematics,  University of Ljubljana \& Institute of Mathematics, Physics and Mechanics, Ljubljana, Slovenia (e-mail: igor.klep@fmf.uni-lj.si).}
\address[Second]{LAAS-CNRS, Université de Toulouse, CNRS, IMT, Toulouse, France (e-mail: vmagron@laas.fr).}
\address[Third]{LAAS-CNRS, Université de Toulouse, CNRS,  Toulouse, France (e-mail: gael.masse@protonmail.com)}
\address[Fourth]{Department of Mathematics, Drexel University, Pennsylvania (e-mail: jurij.volcic@drexel.edu)}       

\begin{abstract}                % Abstract of not more than 250 words.
This work focuses on minimizing the eigenvalue of a noncommutative polynomial subject to a finite number of noncommutative polynomial inequality constraints. 

Based on the Helton-McCullough Positivstellensatz, the noncommutative analog of Lasserre's moment-sum of squares hierarchy provides a sequence of lower bounds converging to the minimal eigenvalue, under mild assumptions on the constraint set. Each lower bound can be obtained by solving a 
semidefinite program. 

We derive complementary converging hierarchies of upper bounds. They are noncommutative analogues of the upper bound hierarchies due to Lasserre for minimizing polynomials over compact sets. Each upper bound can be obtained by solving a 
generalized eigenvalue problem. 
\end{abstract}


\end{frontmatter}
%===============================================================================

\section{Introduction}
\label{sec:intro}

In this work we consider hierarchies of upper bounds for minimal eigenvalue of noncommutative polynomials over noncommutative \emph{basic semialgebraic sets}, i.e., sets defined by finitely many polynomial inequalities.  
We are in particular interested in Bell inequalities, initially introduced by  \cite{bell1964einstein}, that can be viewed as specific types of inequalities on eigenvalues of noncommutative polynomials; see  \cite{pironio2010convergent}. 
In the commutative setting, \emph{polynomial optimization} aims at finding the minimum of a polynomial objective function under finitely many polynomial inequality constraints. 
As shown, e.g., in \cite{laurent2009sums}, this optimization problem is NP-hard to solve exactly, thus a plethora of approximation schemes have been developed in the last two decades, in particular \emph{the moment-sum of squares} (moment-SOS) hierarchy by \cite{lasserre2001global}, also known as the Lasserre hierarchy, that relies on the Positivstellensatz by \cite{putinar1993positive}. 
At a given step of this hierarchy, the corresponding lower bound is computed by solving a semidefinite program, i.e., by minimizing a linear objective function under linear matrix inequality constraints; see \cite{vandenberghe1996semidefinite}. 
The Lasserre hierarchy of lower bounds is ensured to converge to the polynomial minimum under mild natural assumptions often satisfied in practice, e.g., in the presence of a ball constraint. 
Similarly, minimal eigenvalues of noncommutative polynomials can be approximated by a lower bound hierarchy, 
also known as the Navascu{\'e}s-Pironio-Ac{\'\i}n (NPA) hierarchy; 
see \cite{doherty2008quantum,navascues2008convergent,burgdorf2016optimization}, that relies on the Positivstellensatz by \cite{helton2004positivstellensatz}. 
Convergence is ensured under the same assumption as in the commutative case. 

Back in the commutative setting, another hierarchy proposed in \cite{lasserre2011new} yields a monotone sequence of \emph{upper bounds} which converges to the minimum of a polynomial on a given set, and therefore can be seen as complementary to the standard Lasserre hierarchy of lower bounds. 
At a given step of this hierarchy, the corresponding upper bound is computed by solving a so-called {\em generalized eigenvalue} problem. 
While there is no empirical evidence that the Lasserre hierarchy of upper bounds could outperform classical numerical schemes such as brute-force sampling methods based on Monte-Carlo, local optimization solvers based on gradient descent, it turns out that the asymptotic behavior of the upper bound hierarchy has been better understood than for the lower bound hierarchy. 
In \cite{deKlerk16}, the authors obtain convergence rates which often match practical experiments and are no worse than $O(1/\sqrt{d})$, where $d$ is the relaxation order in the hierarchy. 
On some specific sets this convergence rate has been improved to $O(1/d^2)$, e.g., for the box $[-1,1]^n$ by \cite{de2017improved} and for the sphere by \cite{de2020convergence}. 
Recently, similar convergence rates could be obtained by  \cite{slot2022sum} for the standard hierarchy of lower bounds  by combining upper bound rates with an elegant use of Christoffel-Darboux (CD) kernels; see \cite{lasserre2022christoffel} for a recent survey on CD kernels. 
As for the lower bound hierarchy, the sizes of the involved matrix optimization variables are critical and restrict its use to small size problems. 
For the lower bound hierarchy, a common workaround consists of exploiting the structure, e.g., sparsity or symmetry of the input polynomials; see \cite{magron2023sparse} for a recent survey on sparsity-exploiting techniques and \cite{klep2023} for even more sophisticated structure exploitation techniques applied to Bell inequalities. 
A first attempt to improve practical efficiency of the upper bound hierarchy for polynomial optimization has been done in \cite{lasserre2021connecting}.
The idea is to use the pushforward measure of the uniform measure by the polynomial to be minimized. 
In doing so one reduces the initial problem to a related univariate problem and as a result one obtains another hierarchy of upper bounds -- again from generalized eigenvalue problems -- which involves univariate sums of squares polynomials of increasing degree. 

By contrast with the commutative setting, obtaining upper bounds for minimal eigenvalues of noncommutative polynomials can be much more challenging. 
Existing methods include the density matrix renormalization group, e.g., by \cite{white1992density}, which is a numerical variational technique devised to obtain the low-energy physics of quantum many-body systems, or quantum variants of Monte-Carlo methods, e.g., by  \cite{nightingale1998quantum}. 
A first attempt has been done by \cite{ricou2020necessary} to compute minimal eigenvalues of pure quartic oscillators, but without any convergence guarantees and lack of scalability. 

\paragraph*{\textbf{Contributions}.} 
The goal of this work is to propose an alternative to these two families of methods, with potential applications to maximum violation level estimates for Bell inequalities.
We derive two converging upper bound hierarchies for minimal eigenvalues of noncommutative polynomials in separable C*-algebras. 
These hierarchies can be seen as the noncommutative analogues of \cite{lasserre2011new} and \cite{lasserre2021connecting}. 
Similarly to the commutative case, the hierarchies are parametrized by the choice of a sequence of \emph{separating states}, and each upper bound is obtained by solving a single finite-dimensional generalized eigenvalue problem. 
Our framework directly applies to approximate violation levels of Bell inequalities by considering tensor products of universal group C*-algebras with separating state sequences, that can be evaluated using \cite{collins06} calculus for Haar integration over unitary groups. 

\section{Positive polynomials and separating states}
\label{sec:state}

Let $F$ be a noncommutative polynomial in $m$ variables. We are interested in optimizing or deciding positive semidefiniteness of $F(X_1,\dots,X_m)$ over all tuples of operators $(X_1,\dots,X_m)$ satisfying given polynomial relations. Such operators can be often seen as representations of a single (typically very large) operator algebra $\cA$, and the positivity of $F$ on such operators is then equivalent to positivity of a single element $f\in\cA$. For example, consider the problem of whether $F(U_1,\dots,U_n)$ is positive semidefinite for all tuples of unitaries $U_1,\dots,U_n$ acting on a separable Hilbert space. This is equivalent to $f=F(W_1,\dots,W_n)$ being positive semidefinite, where $W_1,\dots,W_n$ are the unitary generators of the universal group C*-algebra $\Cstarfull(\Z^{\star n})$.
Thus we develop our approach to noncommutative positivity eigenvalue optimization in terms of positivity of elements in operator algebras.

Let $\cA$ be a unital C*-algebra. Let us introduce some terminology pertaining to states (positive unital linear functionals) on $\cA$ and $\star$-subalgebras of $\cA$ that is used in this section.
Given two states $\phi$ and $\phi$ on $\cA$, we say that $\psi$ \emph{weakly dominates} $\phi$ if there exists a constant $\alpha>0$ such that $\phi(aa^*)\le\alpha \psi(aa^*)$ for every $a\in\cA$.
A set of states $S$ on $\cA$ is \emph{separating} if for every nonzero $a\in\cA$ there exists $\phi\in S$ such that $\phi(aa^*)>0$. A separating sequence $(\phi_d)_{d=1}^\infty$ is \emph{increasing} if $\phi_{d+1}$ weakly dominates $\phi_d$ for all $d\in\N$.
If $(\phi_d)_{d=1}^\infty$ is a separating sequence, then 
\begin{equation}\label{e:inc}
\left(\frac{2^d}{2^d-1}\sum_{i=1}^d\frac{1}{2^i}\phi_i\right)_{d=1}^\infty
\end{equation}
is an increasing separating sequence,
and $\phi=\sum_{d=1}^\infty\frac{1}{2^d}\phi_d$ is a faithful state on $\cA$.
Note that separable C*-algebras (in particular, finitely generated C*-algebras) always admit faithful states \cite[Exercise I.9.3, or proof of Theorem I.9.23]{Tak02}.
Given $G\subset \cA$ let $\C\langle G\rangle_d$ denote the span of all $\star$-words in $G$ (i.e., products of elements of $G$ and their adjoints) of length at most $d$, and let $\C\langle G\rangle$ denote the $\star$-algebra generated by $G$.
We say that $G$ is \emph{generating} if $\cA$ is the closure in the strong operator topology of $\C\langle G\rangle$.

\begin{theorem}\label{t:general}
Let $\cA$ be a unital C*-algebra, $G$ its generating set, and 
$(\phi_d)_{d=1}^\infty$ an increasing separating sequence of states on $\cA$. 
For $f=f^*\in\C\langle G\rangle$, the following are equivalent:
\begin{enumerate}
    \item[(i)] $f\succeq0$ in $\cA$;
    \item[(ii)] for every $d\in\N$, $\phi_d(hfh^*)\ge0$ for all $h\in\C\langle G\rangle_d$;
    \item[(iii)] for every $d\in\N$, $\phi_d(p(f)fp(f)^*)\ge0$ for all $p\in\R[t]_d$.
\end{enumerate}
\end{theorem}

\begin{proof}
Denote $\phi=\sum_{d=1}^\infty\frac{1}{2^d}\phi_d$. Then $\phi$ is a faithful state on $\cA$.
Clearly (i) implies (ii) and (iii).

Assume (ii) holds. Then $\phi(hfh^*)\ge0$ for all $h\in\C\langle G\rangle$ since $(\phi_d)_d$ is increasing. Since $G$ is generating, we have $\phi(afa^*)\ge0$ for all $a\in\cA$. Let $\pi:\cA\to\cB(\cH)$ be the cyclic $*$-representation of $\cA$ induced by $\phi$ by the Gelfand-Naimark-Segal construction \cite[Theorem 9.14]{Tak02}. Then $\pi$ is a $*$-embedding since $\phi$ is faithful, and $\pi(f)\succeq0$ in $\cB(\cH)$. Therefore $f\succeq0$ in $\cA$ by \cite[Proposition I.4.8 and Theorem I.6.1]{Tak02}.

Assume (iii) holds. Then $\phi(p(f)fp(f)^*)\ge0$ for all $p\in\R[t]$  since $(\phi_d)_d$ is increasing. Let $\cB$ be the unital abelian C*-subalgebra in $\cA$ generated by $f$. By the proof (ii)$\Rightarrow$(i) (with $\cB$ and $\{f\}$ in place of $\cA$ and $G$, respectively), $f\succeq0$ in $\cB$. Therefore $f=bb^*$ for some $b\in\cB$, so $f\succeq0$ in $\cA$.
\end{proof}

\section{Converging upper bound hierarchies}
\label{sec:hierarchy}

Let $\cA$ be a unital C*-algebra with a finite generating set $G$. Impose an order on $G$, and let $G_d$ be the list of $\star$-words in $G$ of length at most $d$, ordered degree-lexicographically. 
To a state $\phi$ on $\cA$, $d\in\N$ and $f=f^*\in\cA$ we assign the matrix
$$\M_{G,d}(f\, \phi):=
\Big(\phi(u^*fv)\Big)_{u,v\in G_d}.
$$
In the special case $G=\{f\}$, write
$$\M_{k,d}(f\, \phi):=\M_{\{f\},d}(f^k\, \phi)=
\Big(\phi(f^{i+j+k})\Big)_{i,j=0}^d
$$
for $k\ge0$.

The minimal eigenvalue of $f$ is denoted by $f_{\min} = \sup\{\alpha\in\R\colon f-\alpha1\succeq0\}$. 

\subsection{Hierarchies of generalized eigenvalue problems}
\label{sec:hierarchies}

\begin{corollary}\label{c:general}
Let $\cA$ be a unital C*-algebra, $G$ its generating set, and 
$(\phi_d)_{d=1}^\infty$ an increasing separating sequence of states on $\cA$. 
For $f=f^*\in\C\langle G\rangle$ and $d\in\N$ denote
\begin{align*}
\lambda_d&=\max\left\{\lambda\in\R\colon 
\M_{G,d}(f\, \phi_d)\succeq \lambda \M_{G,d}(1\, \phi_d)
\right\}, \\
\eta_d&=\max\left\{\eta\in\R\colon \M_{1,d}(f\, \phi_d)\succeq \eta \M_{0,d}(f\, \phi_d)\right\}.
\end{align*}
Then $(\lambda_d)_d$ and $(\eta_d)_d$ are decreasing sequences, and
$$\lim_{d\to\infty}\lambda_d=
\lim_{d\to\infty}\eta_d=f_{\min}.
$$
\end{corollary}

\begin{proof}
Monotonicity of $(\lambda_d)_d$ and $(\eta_d)_d$ is a consequence of $(\phi_d)_d$ being an increasing sequence of states. The limit values follow from Theorem \ref{t:general}.
\end{proof}

The sequences $(\lambda_d)_d$ and $(\eta_d)_d$ can be viewed as the noncommutative analogues of the sequences of upper bounds for standard polynomial optimization from \cite{lasserre2011new} and \cite{lasserre2021connecting}, respectively. 
At a given relaxation order $d$ computing either $\lambda_d$ or $\eta_d$ boils down to solving a generalized eigenvalue problem. 

Given a self-adjoint element $f$ of a finitely generated C*-algebra $\cA$, Corollary \ref{c:general} and \eqref{e:inc} give a sequence of generalized eigenvalue problems whose solutions converging to the infimum of $f$ in $\cA$, as long as there is an increasing sequence of states on $\cA$ that is efficiently computable. 
The following are examples of separable C*-algebras and their faithful states, or separating sequences of states. 

\begin{enumerate}
    \item $\Cstarred(G)$ for a finitely generated discrete group $G$, with the canonical tracial state $\tau$.
    \item $\Cstarfull(\Z^{\star n})$, with a separating sequence
    \begin{equation}\label{e:integralstate}
    \phi_d(w)=\frac{1}{d}\int_{U\in \unit{d}^n}\tr w(U)\,{\rm d}U.
    \end{equation}
    The separating property of \eqref{e:integralstate} follows by \cite[Theorem 7]{choi80} (cf. \cite[Corollary 4.7]{ratresolvable}).
    Note that the states \eqref{e:integralstate} do not readily form an increasing sequence, but their finite combinations as in \eqref{e:inc} do. The states \eqref{e:integralstate} can be evaluated using the Collins-\'Sniady calculus for Haar integration over unitary groups \cite[
Corollary 2.4]{collins06}.
    \item Suppose $\cA_1$ and $\cA_2$ are C*-algebras with faithful states $\phi_1$ and $\phi_2$ respectively. Then the state $\phi_1\otimes\phi_2$ on the minimal (injective) tensor product $\cA_1\otimes_{\min} \cA_2$ is faithful \cite[Theorem IV.4.9]{Tak02}, and the state $\phi_1\star\phi_2$ on the reduced free product $\cA_1\star\cA_2$ is faithful \cite[Theorem 1.1]{dykema98}.
    Values of $\phi_1\otimes\phi_2$ and $\phi_1\star\phi_2$ are easily expressible with values of $\phi_1$ and $\phi_2$.
    \item Combining (2) and (3), 
    one gets an explicit separating sequence for $\Cstarfull(\Z^{\star n})\otimes_{\min} \Cstarfull(\Z^{\star m})$.
    As a side remark, note that $\Cstarfull(\Z^{\star n})\otimes_{\min} \Cstarfull(\Z^{\star m})$ is not isomorphic to $\Cstarfull(\Z^{\star n})\otimes_{\max} \Cstarfull(\Z^{\star m})\cong \Cstarfull(\Z^{\star n}\times \Z^{\star m})$ for $n,m\ge2$ by the refutation of Connes' embedding conjecture (and its equivalence to Kirchberg's conjecture).
\end{enumerate}

\subsection{Bell inequalities}
\label{sec:bell}
Now we apply the above framework to obtain lower bounds for maximal violation levels for Bell inequalities. 
One particularly famous Bell inequality is the CHSH inequality by \cite{clauser1969proposed}, where the setting is a quantum system
consisting of two measurements for each party, each with the two outcomes $\pm 1$. 
The measurements can be modeled by four unitary operators $x_1, x_2, y_1, y_2$ satisfying $x_i^2 = 1 = y_j^2$. 
Since we are interested in the non-local behavior of our quantum system, we impose the additional constraint that the operators 
$x_i$'s act on one Hilbert space, and $y_j$'s act on another Hilbert space.
The maximum violation of CHSH corresponds to  the opposite of the minimal eigenvalue of $f = -x_1 \otimes y_1 - x_1 \otimes y_2 - x_2 \otimes y_1 + x_2 \otimes y_2$ (acting on the tensor product of Hilbert spaces) under the above unitary/commutativity constraints. 
In the sequel, we denote this minimal eigenvalue by $f_{\min}$.

For certain Bell inequalities the measurement operators are not initially constrained to be unitaries (as in the above example) but to be  projectors, in which case we apply a change of variables to obtain unitaries, by defining $b_i := 2 x_i -1$ and $c_j := 2 y_j -1$, which yields $b_i^2 = (2 x_i -1)^2 = 1 = c_j^2$, for each $i \in [n], \ j \in [m]$. 
If the operators are initially unitaries, then we take $b_i := x_i$ and $c_j := y_j$. 
Then the polynomials involved in Bell inequalities lie in the  separable C*-algebra $\Cstarfull(\Z^{\star n})\otimes_{\min} \Cstarfull(\Z^{\star m})$. When $m,n\le2$, this is isomorphic to $\Cstarred(\Z_{2}^{\star n} \times \Z_2^{\star m})$ since $\Z_{2}^{\star n} \times \Z_2^{\star m}$ is amenable \cite[Theorems XIII.4.6 and XIII.4.7]{Tak03}. 
For any nonzero $s\in \N$, let $\I_s$ be the identity matrix of size $s$. 
Our strategy 
to obtain upper bounds of $f_{\min}$ 
is to rely on tensor products of separating sequences~\eqref{e:integralstate} from  Section~\ref{sec:hierarchies} by parametrizing Hermitian unitaries by unitaries and signatures, i.e., by writing each Hermitian unitary $b_i$ of size $d$ as $b_i = U_i \begin{pmatrix}
\I_{r_i} & 0 \\
0 & - \I_{d-{r_i}}
\end{pmatrix}  U_i^{\star}$ for some $r_i \leq d$ and $U_i \in \unit{d}$. 
It turns out that it is sufficient to consider only $b_i$ of even size $2d$ with $r_i=d$.
Then for every word $w$ in $b$, one could use the state returning 
\begin{align*}
%\label{e:integralstateherm1}
\frac{1}{2d}\int_{U\in \unit{2d}^n}\tr & \left[
w \Biggl(U_1 \begin{pmatrix}
\I_d & 0 \\
0 & - \I_d
\end{pmatrix}  U_1^{\star}, \right.\\
 & \left. \dots, U_n \begin{pmatrix}
\I_d & 0 \\
0 & - \I_d
\end{pmatrix}  U_n^{\star}   \Biggr)
\right]
\,{\rm d}U.
\end{align*}
Since $\tr(w_1\otimes w_2)=\tr(w_1)\tr(w_2)$ for words $w_1$ in the $b_i$'s and words $w_2$ in the $c_j$'s,
we simply rely on products of such state evaluations for our numerical experiments. 

As preliminary computation outcomes based on the \texttt{IntU} Mathematica library by \cite{puchala2011symbolic}, the upper bound sequences obtained for CHSH are $(\lambda_1, \lambda_2) = (0.146, -0.016)$ and $(\eta_1,\eta_2) = (0, -0.066)$. 
Since the quantum bound is known to be $f_{\min} = (1-\sqrt{2})/2 \simeq  -0.207$, one likely needs to be able to efficiently compute quite a few steps before one gets close to the actual value. 
Further work directions include a more careful algorithmic implementation towards this goal.

%\bibliography{references}             % bib file to produce the bibliography

\begin{thebibliography}{30}
\providecommand{\natexlab}[1]{#1}
\providecommand{\url}[1]{\texttt{#1}}
\providecommand{\urlprefix}{URL }
\expandafter\ifx\csname urlstyle\endcsname\relax
  \providecommand{\doi}[1]{doi:\discretionary{}{}{}#1}\else
  \providecommand{\doi}{doi:\discretionary{}{}{}\begingroup
  \urlstyle{rm}\Url}\fi

\bibitem[{Bell(1964)}]{bell1964einstein}
Bell, J.S. (1964).
\newblock On the {Einstein Podolsky Rosen} paradox.
\newblock \emph{Physics Physique Fizika}, 1(3), 195.

\bibitem[{Burgdorf et~al.(2016)Burgdorf, Klep, and
  Povh}]{burgdorf2016optimization}
Burgdorf, S., Klep, I., and Povh, J. (2016).
\newblock \emph{Optimization of polynomials in non-commuting variables}.
\newblock Springer.

\bibitem[{Choi(1980)}]{choi80}
Choi, M.D. (1980).
\newblock The full {$C\sp{\ast} $}-algebra of the free group on two generators.
\newblock \emph{Pacific J. Math.}, 87(1), 41--48.

\bibitem[{Clauser et~al.(1969)Clauser, Horne, Shimony, and
  Holt}]{clauser1969proposed}
Clauser, J.F., Horne, M.A., Shimony, A., and Holt, R.A. (1969).
\newblock Proposed experiment to test local hidden-variable theories.
\newblock \emph{Physical review letters}, 23(15), 880.

\bibitem[{Collins and \'{S}niady(2006)}]{collins06}
Collins, B. and \'{S}niady, P. (2006).
\newblock Integration with respect to the {H}aar measure on unitary, orthogonal
  and symplectic group.
\newblock \emph{Comm. Math. Phys.}, 264(3), 773--795.

\bibitem[{De~Klerk et~al.(2017)De~Klerk, Hess, and Laurent}]{de2017improved}
De~Klerk, E., Hess, R., and Laurent, M. (2017).
\newblock {Improved convergence rates for Lasserre-type hierarchies of upper
  bounds for box-constrained polynomial optimization}.
\newblock \emph{SIAM Journal on Optimization}, 27(1), 347--367.

\bibitem[{de~Klerk and Laurent(2020)}]{de2020convergence}
de~Klerk, E. and Laurent, M. (2020).
\newblock {Convergence analysis of a Lasserre hierarchy of upper bounds for
  polynomial minimization on the sphere}.
\newblock \emph{Mathematical Programming}, 1--21.

\bibitem[{de~Klerk et~al.(2016)de~Klerk, Laurent, and Sun}]{deKlerk16}
de~Klerk, E., Laurent, M., and Sun, Z. (2016).
\newblock Convergence analysis for {Lasserre's measure-based hierarchy of upper
  bounds for polynomial optimization}.
\newblock \emph{Mathematical Programming A}, 1--30.

\bibitem[{Doherty et~al.(2008)Doherty, Liang, Toner, and
  Wehner}]{doherty2008quantum}
Doherty, A.C., Liang, Y.C., Toner, B., and Wehner, S. (2008).
\newblock The quantum moment problem and bounds on entangled multi-prover
  games.
\newblock In \emph{2008 23rd Annual IEEE Conference on Computational
  Complexity}, 199--210. IEEE.

\bibitem[{Dykema(1998)}]{dykema98}
Dykema, K.J. (1998).
\newblock Faithfulness of free product states.
\newblock \emph{J. Funct. Anal.}, 154(2), 323--329.

\bibitem[{Helton and McCullough(2004)}]{helton2004positivstellensatz}
Helton, J. and McCullough, S. (2004).
\newblock A {Positivstellensatz} for non-commutative polynomials.
\newblock \emph{Transactions of the American Mathematical Society}, 356(9),
  3721--3737.

\bibitem[{Hrga et~al.(2023)Hrga, Klep, and Povh}]{klep2023}
Hrga, T., Klep, I., and Povh, J. (2023).
\newblock Certifying optimality of bell inequality violations: Noncommutative
  polynomial optimization through semidefinite programming and local
  optimization.
\newblock \emph{Accepted for publication in SIAM Journal on Optimization}.

\bibitem[{Klep et~al.(2017)Klep, Vinnikov, and Vol\v{c}i\v{c}}]{ratresolvable}
Klep, I., Vinnikov, V., and Vol\v{c}i\v{c}, J. (2017).
\newblock Null- and {P}ositivstellens{\"a}tze for rationally resolvable ideals.
\newblock \emph{Linear Algebra Appl.}, 527, 260--293.

\bibitem[{Lasserre(2011)}]{lasserre2011new}
Lasserre, J.B. (2011).
\newblock A new look at nonnegativity on closed sets and polynomial
  optimization.
\newblock \emph{SIAM Journal on Optimization}, 21(3), 864--885.

\bibitem[{Lasserre(2001)}]{lasserre2001global}
Lasserre, J.B. (2001).
\newblock Global optimization with polynomials and the problem of moments.
\newblock \emph{SIAM Journal on optimization}, 11(3), 796--817.

\bibitem[{Lasserre(2021)}]{lasserre2021connecting}
Lasserre, J.B. (2021).
\newblock Connecting optimization with spectral analysis of tri-diagonal
  matrices.
\newblock \emph{Mathematical Programming}, 190(1), 795--809.

\bibitem[{Lasserre et~al.(2022)Lasserre, Pauwels, and
  Putinar}]{lasserre2022christoffel}
Lasserre, J.B., Pauwels, E., and Putinar, M. (2022).
\newblock \emph{The Christoffel--Darboux Kernel for Data Analysis}, volume~38.
\newblock Cambridge University Press.

\bibitem[{Laurent(2009)}]{laurent2009sums}
Laurent, M. (2009).
\newblock Sums of squares, moment matrices and optimization over polynomials.
\newblock \emph{Emerging applications of algebraic geometry}, 157--270.

\bibitem[{Magron and Wang(2023)}]{magron2023sparse}
Magron, V. and Wang, J. (2023).
\newblock \emph{Sparse polynomial optimization: theory and practice}.
\newblock World Scientific.

\bibitem[{Navascu{\'e}s et~al.(2008)Navascu{\'e}s, Pironio, and
  Ac{\'\i}n}]{navascues2008convergent}
Navascu{\'e}s, M., Pironio, S., and Ac{\'\i}n, A. (2008).
\newblock A convergent hierarchy of semidefinite programs characterizing the
  set of quantum correlations.
\newblock \emph{New Journal of Physics}, 10(7), 073013.

\bibitem[{Nightingale and Umrigar(1998)}]{nightingale1998quantum}
Nightingale, M.P. and Umrigar, C.J. (1998).
\newblock \emph{Quantum Monte Carlo methods in physics and chemistry}.
\newblock 525. Springer Science \& Business Media.

\bibitem[{Pironio et~al.(2010)Pironio, Navascu{\'e}s, and
  Acin}]{pironio2010convergent}
Pironio, S., Navascu{\'e}s, M., and Acin, A. (2010).
\newblock Convergent relaxations of polynomial optimization problems with
  noncommuting variables.
\newblock \emph{SIAM Journal on Optimization}, 20(5), 2157--2180.

\bibitem[{Pucha{\l}a and Miszczak(2017)}]{puchala2011symbolic}
Pucha{\l}a, Z. and Miszczak, J.A. (2017).
\newblock Symbolic integration with respect to the {Haar} measure on the
  unitary group.
\newblock \emph{Bull. Pol. Acad. Sci.-Tech. Sci.}, 65(1), 21--27.

\bibitem[{Putinar(1993)}]{putinar1993positive}
Putinar, M. (1993).
\newblock Positive polynomials on compact semi-algebraic sets.
\newblock \emph{Indiana University Mathematics Journal}, 42(3), 969--984.

\bibitem[{Ricou(2020)}]{ricou2020necessary}
Ricou, A. (2020).
\newblock \emph{Necessary conditions for nonnegativity on*-algebras and ground
  state problem}.
\newblock Master's thesis, National University of Singapore (Singapore).

\bibitem[{Slot(2022)}]{slot2022sum}
Slot, L. (2022).
\newblock Sum-of-squares hierarchies for polynomial optimization and the
  christoffel--darboux kernel.
\newblock \emph{SIAM Journal on Optimization}, 32(4), 2612--2635.

\bibitem[{Takesaki(2002)}]{Tak02}
Takesaki, M. (2002).
\newblock \emph{Theory of operator algebras. {I}}, volume 124 of
  \emph{Encyclopaedia of Mathematical Sciences}.
\newblock Springer-Verlag, Berlin.
\newblock Reprint of the first (1979) edition, Operator Algebras and
  Non-commutative Geometry, 5.

\bibitem[{Takesaki(2003)}]{Tak03}
Takesaki, M. (2003).
\newblock \emph{Theory of operator algebras. {III}}, volume 127 of
  \emph{Encyclopaedia of Mathematical Sciences}.
\newblock Springer-Verlag, Berlin.

\bibitem[{Vandenberghe and Boyd(1996)}]{vandenberghe1996semidefinite}
Vandenberghe, L. and Boyd, S. (1996).
\newblock Semidefinite programming.
\newblock \emph{SIAM review}, 38(1), 49--95.

\bibitem[{White(1992)}]{white1992density}
White, S.R. (1992).
\newblock Density matrix formulation for quantum renormalization groups.
\newblock \emph{Physical review letters}, 69(19), 2863.

\end{thebibliography}
                                                     % with 

%
\end{document}